\documentclass[11pt]{article}

\usepackage{amssymb,amsmath,amsthm,mathtools}
\usepackage{graphicx,fullpage}
\usepackage{longtable}
\usepackage[ruled,linesnumbered]{algorithm2e}

\usepackage[shortlabels]{enumitem}
\setlist{
  listparindent=3em,
  itemindent=\parindent,
  parsep=0pt,
  leftmargin=\parindent
}

\newlist{nested}{enumerate}{5}
\setlist[nested]{
  nosep,    
  noitemsep,
  listparindent=2\parindent,
    leftmargin=2\parindent,
  parsep=0pt
}

\usepackage{caption}
\usepackage{color,soul}
\usepackage[table]{xcolor}
\usepackage{verbatim}
\usepackage{comment}

\usepackage{tikz-cd}

\usepackage{hyperref}
\PassOptionsToPackage{unicode}{hyperref}

\newcommand{\old}[1]{}
\mathchardef\mhyphen="2D 
\newtheorem{theorem}{Theorem}

\newtheorem{cor}[theorem]{Corollary}

\newtheorem{dfn}[theorem]{Definition}

\newcounter{cases}
\newcounter{subcases}[cases]

\theoremstyle{definition}
\newtheorem{example} {Example}

\DeclarePairedDelimiter\floor{\lfloor}{\rfloor}

\def\etal{{\it et~al.}\,}

\begin{document}

\title{Improved Bounds for Permutation Arrays Under Chebyshev Distance} 
\author{Sergey Bereg,\thanks{
Department of Computer Science,
University of Texas at Dallas,
Box 830688,
Richardson, TX 75083
USA.}
\and Mohammadreza Haghpanah,$^*$
\and Brian Malouf,$^*$
\and I. Hal Sudborough$^*$
}

\maketitle

\begin{abstract}
Permutation arrays under the Chebyshev metric have been considered for error correction in noisy channels. 
Let $P(n,d)$ denote the maximum size of any array of permutations on $n$ symbols with pairwise Chebyshev distance $d$. 
We give new techniques and improved upper and lower bounds on $P(n,d)$, including a precise formula for $P(n,2)$.
\end{abstract}


\section{Introduction}
\label{sec:intro}

In \cite{Klove10} an  interesting study of permutation arrays under the Chebyshev metric was presented. This complemented many studies of permutation arrays under other metrics, such as the Hamming metric  \cite{bls-18} \cite{bmmms-19}
\cite{chu2004}, Kendall $\tau$  metric \cite{Jiang} \cite{buzaglo}, and several others \cite{Deza2}.
The use of the Chebyshev metric was motivated by applications of error correcting codes and recharging in flash memories \cite{Jiang}.

Let $\sigma$ and $\pi$ be two permutations (or strings) over an alphabet $\Sigma \subseteq [1 ... n] = \{1,2, ... , n\}$.
The Chebyshev distance between $\sigma$ and $\pi$, denoted by $d(\sigma,\pi)$,
is $\max \{ ~|\sigma(i) - \pi(i)|~ | ~ i \in \Sigma ~\} $. For an array (set) A of permutations (strings), the pairwise Chebyshev distance of $A$, denoted by $d(A)$, is $\min \{ ~d(\sigma,\pi)~ |~ \sigma, \pi \in A ~ \}$. An array A of permutations on $[1...n]$ with $d(A)=d$ will be called an $(n,d)$ PA. Note that this includes the case when 
A is a set of integers, $i.e.$ a set of strings of length one, where $d(A)$ corresponds to the minimum difference between integers in the set.
Let $P(n,d)$ denote the maximum cardinality of any $(n,d)$-PA $A$. 
More generally, let $P_d(\Sigma)$ denote the maximum cardinality of any array of permutations over the alphabet $\Sigma \subseteq [1 ... n]$ with Chebyshev distance $d$.
For example, $P_2(\{1,3,5,7\})$ = 4! = 24, whereas $P(4,2)=6$.

We present several methods to improve on lower and upper bounds for $P(n,d)$. For comparison, 
we begin with the following theorem from \cite{Klove10}.

\begin{theorem}
\label{recursion2}(\cite{Klove10})
If $n > d \ge 1$, then
$P(n + 1,d) \ge ( \floor{\frac{n}{d}} + 1)  P(n,d)$.
\end{theorem}

To generalize, let $A$ be a subset of $[1...(n+1)]$ such that $d(A) \ge d$, then, for all $i \in A$, $P_d([1..(n+1)]-\{i\}) \ge P(n,d)$. Observe that the set $\{1,d+1,2d+1, ... , \floor{\frac{n}{d}}d+1\}$ is a subset of $[1...(n+1)]$ with $\floor{\frac{n}{d}}+1$ elements with Chebyshev distance
d and was used in \cite{Klove10} to prove Theorem \ref{recursion2}. 

\begin{theorem} 
\label{1step}
Let $A$ be a subset of $[1...(n+1)]$ such that $d(A) \ge d$.
If $n > d \ge 1$, then
$P(n + 1,d) \ge \sum_{i \in A}   P_d([1..(n+1)]-\{i\})$.
\end{theorem}

Theorem \ref{1step} is a generalization of Theorem \ref{recursion2} and often gives improved lower bounds. For example, using Theorem \ref{recursion2}, one obtains $P(11,3) \ge 36,132$, as $\floor{\frac{10}{3}}+1=4$ and the best lower bound currently known for $P(10,3)$ is 9,033. 
Using Theorem \ref{1step} and choosing $A=\{3,6,9\}$, one obtains the lower bound 53,549, as $P_3([1..11]-\{3\}) = P_3([1..11]-\{9\} \ge 17,573$ and $P_3([1..11]-\{6\}) \ge 18,403$.

Another recursive technique in \cite{Klove10} gave the following result.

\begin{theorem}
\label{recursion}
(\cite{Klove10})
 If $n>d$ and $r \ge 2$, then
$P(rn,rd) \ge P(n,d)^r$.
\end{theorem}

For example, we use Theorem \ref{recursion} to get $P(18,4) \ge P(9,2)^2 = 2,520^2 = 514,382,400$.
Theorem \ref{recursion} is generalized by Theorem \ref{general}, which subsumes Theorem \ref{recursion} and gives several new lower bounds.
For example, we use Theorem \ref{general}, with a=3, to get $P(18,5) \ge P(11,3)*P(7,2) \ge 53,549 * 630 = 33,735,870$.

\begin{theorem} \label{general}
$P(n,d) \ge \max\{P(n_1,d_1)\cdot P(n_2,d_2)~|~ d_1+d_2=d$ and $n_1+n_2=n$ and, for some constant $a$, $n_1 = ad_1+r_1$ and $n_2 = ad_2+r_2$, with $0 \le~r_1 \le~d_1$ and
with $0 \le~r_2 \le~d_2\}$, where the maximum is taken over all possible values of ${n_1,n_2,d_1,d_2}$.
\end{theorem}

As another example, we use Theorem \ref{general} to get the lower bound $P(16,9) \ge P(9,5)*P(7,4) \ge 3,399$, where $a=1$, $9=1*5+4$, $7=1*4+3$, and the best lower bounds known for $P(9,5)$ and $P(7,4)$ are 103 and 33, respectively.

For given $n$ and $d$, Klove et al \cite{Klove10} defined $C=\{ (\pi_1, ... , \pi_n) \in S_n ~|~ \pi_i = i \mod{d}$, for all $i \in [1..n] \}$ and gave the following theorems:

\begin{theorem}(\cite{Klove10})
If $n = ad+b$, where $0 \le b < d$, then $C$
is an $(n,d)$ PA
and $\vline C \vline = ((a+1)!)^b(a!)^{d-b} $.
\end{theorem}

\begin{theorem}(\cite{Klove10})
\label{mod}
If $n = ad+b$, where $0 \le b < d$, then $P(n,d) \ge ((a+1)!)^b(a!)^{d-b}$.
\end{theorem}

Klove et al \cite{Klove10} gave, as an example, the lower bound $P(2a,2) \ge (a!)^2$. 
They also gave the improvement, using Theorem \ref{recursion2} iteratively, $P(2a,2) \ge \frac{97}{24}(a!)^2$. We give an exact equation for $P(n,2)$. 
Specifically, we show $P(2a,2) = \frac{(2a)!}{2^a}$. 

\begin{theorem} \label{Two}
$P(n,2) = \frac{n!}{2^{\floor{n/2}}}$.
\end{theorem}

The iterative use of Theorem \ref{recursion2} can be improved further by a generalization of Theorem \ref{1step} using strings of more than one symbol. Let $A$ be a set of length $m$ strings with no repeated symbols (permutations) over $[1..(n+m)]$ with $d(A) \ge d$. 
By an abuse of notation, for each $\sigma \in A$, let $\sigma^C$ denote the complement in $[1..(n+m)]$ of the set of symbols used in $\sigma$. 
As in Theorem \ref{1step}, we show that $P(n + m,d) \ge \sum_{\sigma \in A}   P_d(\sigma^C)$. 
Let $Q((n+m),m,d)$ denote the collection of all sets A of  permutations on a $m$ symbol subset of $[1..(n+m)]$ with $d(A) \ge d$. 
Maximizing the sum over all such sets A yields the following.

\begin{theorem}
\label{prefix}
For any $n\ge d\ge 1, m\ge 1$,
$P(n+m,d) \ge \max_{A \in Q((n+m),m,d)}~ \sum_{\sigma \in A} P_d(\sigma^C)$.
\end{theorem}

In \cite{Klove10} a 3-fold iterative use of 
Theorem \ref{recursion2}, for $d=3$ and $n=5$ gives a set $S \in Q(8,3,3)$ with $\vline S \vline = 18$. 
That is, 
$(\floor{\frac{5}{3}} + 1)( \floor{\frac{6}{3}} + 1)( \floor{\frac{7}{3}} + 1)$ = 18. However, by computation one can obtain a set $T \in Q(8,3,3)$ with $\vline T \vline = 24$. 
Thus, not only can one obtain a larger subset of $[1..(n+m)]$ than the iterative use of Theorem \ref{recursion2}, but also larger sets than $P(n,d)$ by the use of complement alphabets. 
For $m<n$, let $P(n,m,d)$ denote
the maximum cardinality of any set $A$ in $Q(n,m,d)$. 
We have computed several lower bounds for
$P(n,m,d)$. See, for example, Tables \ref{Pnm2-3} and \ref{Pnm4-5} in Section \ref{pref}.

\begin{cor}
\label{prefix2}
For any $n\ge d\ge 1, m\ge 1$,
    $P(n+m,d) \ge P(n+m,m,d) * P(n,d)$.
\end{cor}

\begin{proof}
That is, for any set $A \in Q((n+m),m,d)$, and any $\sigma \in A, P_d(\sigma^C) \ge P(n,d)$.   
\end{proof}



We have shown in previous examples that Corollary \ref{prefix2} gives improved lower bounds, by computation, over an iterative use of Theorem \ref{recursion2}. The next theorem show that such improvements  exist even for arbitrarily large n. For example, if $d=5$ and $k=2$, an iterative use of Theorem \ref{recursion2} gives $P(dk+d-1,d) = P(14,5) \ge 
(\floor {\frac{13}{5}}+1)
(\floor {\frac{12}{5}}+1)
(\floor {\frac{11}{5}}+1)(\floor {\frac{10}{5}}+1)
P(10,5) =
3^4P(10,5) = 81P(10,5)$. By Theorem \ref{gen}, $P(dk+d-1) = P(14,5) \ge (3^5 - \binom{6}{4})P(10,5) = 228P(10,5)$.


\begin{theorem}
\label{gen}
For any $d\ge 3$ and $k\ge 1$, 
$P(dk+d-1,d)\ge \left((k+1)^d - {k+d-1\choose d-1}\right) P(dk-1,d).$
\end{theorem}

As another example of the improvement shown by Theorem \ref{gen} consider the case when $k=3$ and $d=3$. 
The theorem states that $P(11,3) \ge 54 \cdot P(8,3)$, whereas the three fold iterative use of Theorem \ref{recursion2} gives $P(11,3) \ge (\floor{\frac{10}{3}}+1 )\cdot(\floor{\frac{9}{3}}+1)\cdot (\floor {\frac{8}{3}}+1)\cdot P(8,3) = 48\cdot P(8,3)$. By computational methods, we show that $P(11,3,3) \ge 59$ and hence, by Theorem \ref{prefix},
we have $P(11,3) \ge 59 \cdot P(8,3)$. 
In fact, as shown in Table \ref{LowerBounds}, $P(11,3) \ge 53,549$.

 Let $V(n,d)$ be the number of permutations on $\{1,2,\dots,n\}$ within distance $d$ of the identity permutation.

Kl{\o}ve \etal \cite{Klove10} also gave general lower and upper bounds.

\begin{theorem}
\cite{Klove10}
    For $n>d \ge 2$,
$P (n, d) \ge \frac{n!}{V(n,d-1)}$
 
\end{theorem}

\begin{theorem}
\cite{Klove10}
\label{Bound}
For even $d$ and $2d \ge n \ge d \ge 2$,
$P(n,d) \le \frac{(n+1)!}{V(n+1,d/2)},$ 
\end{theorem}

In Theorem \ref{upperbound} we give a better upper bound.  Using Theorem \ref{upperbound} we show, for example, $P(11,6) \le 462$.
Kl{\o}ve \cite{Klove11} also proved  lower bounds on the size of spheres of permutations under the Chebyshev distance.

\begin{theorem} \label{upperbound}
For $1 \leq k \leq d < n$,
\begin{align*}
P(n,d) \le P(n-k, d) \cdot \binom{n}{k}.    
\end{align*}
\end{theorem}

In \cite{Klove10} there is also the following interesting theorem.

\begin{theorem}\label{fixed} \cite{Klove10}~
For fixed $r$, there exist constants $c_r$ and $d_r$ such that $P(d+r,d) = c_r$, for $d \ge d_r$.
\end{theorem}

Moreover, an upper bound on the constants $c_r$ and $d_r$ is given in \cite{Klove10}. The proof uses the concept of $potent$ symbols. Basically, an integer is potent for Chebyshev distance d if there is another integer, say j, in the given alphabet, such that $\vline j - i \vline \ge d$.  That is, the symbol can be used in permutations to achieve distance d.

\begin{dfn}
If A is a PA on $d+r$ symbols with Chebyshev distance $d$, then the integers  \\
1,2, ... ,r and d+1,d+2, ... ,d+r are potent. 
\end{dfn}

The following theorem provides improved upper bounds for the constants $c_r$ and $d_r$ of Theorem \ref{fixed}.

\begin{theorem} \label{upperbound2}
Suppose that $P(n_0,n_0-k)\le m$ such that 
\begin{equation} \label{condition}
2k(m+1) < (n_0+1)(1+\lfloor n_0/(2k-1)\rfloor).
\end{equation}
Then $P(n,n-k)\le m$, for all $n\ge n_0$.
\end{theorem}

As an example, Theorem \ref{upperbound2} can be used to show that the constants $c_2,d_2$ in Theorem \ref{fixed} are $d_2 = 3$ and $c_2 = 10$. 

\begin{cor}\label{ten}
$P(n,n-2) = 10$, for all $n\ge 5$.
\end{cor}

As part of the proof of  Corollary \ref{ten}, we have computed a PA A on [1..5] with $d(A)=3$, so $P(5,3) \ge 10$. In \cite{Klove10},  $P(5,3) \le 9$ was claimed, but was apparently due to a computational error. 

Theorem \ref{upperbound2} can also be used to show improved bounds for $c_r$ and $d_r$, for $r \ge 3$. For example, by Theorem \ref{upperbound}, we have $P(n,n-3) \le P(n-1,n-3) \cdot
\binom{n}{1} = 10 \cdot n$, for all $n \ge 6$. Observe that, for $k=3$, $n_0=295$, and $m=2950$, the inequality of Equation (1) is true. So, $P(n,n-3) \le 2{,}950$, for all $n \ge 295$. Thus, $c_3 \le 2{,}950$ and $d_3 \le 295$, which improves the bounds $c_3 \le 46{,}080$ and $d_3 \le 230{,}401$ given in \cite{Klove10}.

In \cite{Klove10} a few additional recursive constructions were described to obtain lower bounds for $P(n,d)$. For example, for any permutation $\sigma \in S_n$ and any $m~ (1 \le m \le n)$, define $\phi_m(\sigma) = (m,\pi_1,\pi_2, ... , \pi_n)$, where:

$\pi_i = \sigma_i$,        if~$i < m$, and

$\pi_i = \sigma_{i}+1$,   if~ $i \ge m$.

For any PA $A$ and symbols
$1 \le s_1 < s_2 < ... < s_t \le n+1$, define $A[s_1,s_2, ... , s_t]$ to be $ \{ \phi_m(\sigma)~| ~\sigma \in A,~ m \in \{ s_1,s_2, ... , s_t \} \}$

\begin{theorem} (\cite{Klove10})
\label{n/d}
If $A$ is an $(n,d)$ PA of size $M$ and
$s_j + d \le s_{j+1}$, for $1 \le j \le t$-1,
then $A[s_1,s_2, ... ,s_t]$ is an (n + 1,d) PA of size $tM$.

\end{theorem}

\begin{theorem} (\cite{Klove10})
\label{thm1}
If A is an $(n,d)$ PA of size M and $n \le 2d$, then A[d] is an
(n+1,d+1) PA of size M.
\end{theorem}

Theorem \ref{n/d} implies the following:

\begin{theorem} (\cite{Klove10})
\label{diag}
If $d<n \le 2d$, then
$P (n + 1, d + 1) \ge P (n, d)$.
\end{theorem}

In Table \ref{LowerBounds} we give several lower bounds for $P(n,d)$ and in Table \ref{UpperBounds} we give several upper bounds for $P(n,d)$.

\section{Lower Bounds}

In \cite{Klove10} a greedy algorithm was used to find a PA $C$ on $[1..n]$ with $d(C) \ge d$: 

\begin{quotation}
Let the identity permutation in $S_n$ be the first permutation in $C$. For any set of permutations chosen, choose as the next permutation in $C$ the lexicographically next permutation in $S_n$ with distance at least $d$ to the chosen permutations in C if such a permutation
exists.
\end{quotation}

We modified this greedy algorithm by choosing an initial set $C$ of  pairwise distance d permutations randomly. Because of the randomness, we also allowed the algorithm to automatically start again and repeat the process while recording the best result. 
We call this the $Random/Greedy$ strategy.

Many of the lower bounds in Table \ref{LowerBounds}, for small values of n, were obtained by this modified greedy algorithm. A few were found by computing a largest clique in a graph, whose nodes are all permutations, and edges are between nodes at Chebyshev distance $\ge d$, called the $Clique$ approach. Others were found using Theorems \ref{1step}, \ref{general}, \ref{Two}, or \ref{prefix}.
Computations using the ideas of Theorem \ref{prefix} were often done with a Max Weighted Clique solver tool \cite{NetX} \cite{MaxClique}. That is, to compute a lower bound for $P(n+m,d)$, a graph G was created with nodes labeled by permutations on m symbols of [1..(n+m)], and whose edges connect two nodes with labels $L_1$ and $L_2$, where $d(L_1,L_2) \ge d$. A node with label L is given a weight of $P_d(L^C)$, where the complement is taken with respect to the set [1..(n+m)]. Values for $P_d(L^C)$ were pre-computed, using a modification of the Random/Greedy algorithm. A maximum weighted clique of G corresponds to the lower bound given in Theorem \ref{prefix}. 
As the set of all permutations on a m symbol subset of [1..(n+m)] gets very large as m and n get large, heuristics were sometimes used to decide which permutations to use as labels in the graph G.

\begin{table}[htb]
{\small
\caption {Lower Bounds for $P(n,d)$.} 
\centering
\begin{tabular}{|r |r |r |r |r |r |r |r|r|r|c}
\hline
$n/d$ & 2 & 3 & 4 & 5 & 6 & 7&8&9&10\\ [0.5ex] 
\hline
2 &1&1&1&1&1&1&1&1&1\\ 
\hline
3 &{\bf 3}&1&1&1&1&1&1&1&1\\
\hline
4 & {\bf 6}&{\bf 3}&1&1&1&1&1&1&1\\
\hline
5 & {\bf 30}&{\bf 10}&{\bf 3}&1&1&1&1&1&1 \\
\hline
6 & {\bf 90}&{\bf 20}&{\bf 10}&{\bf 3}&1&1&1&1&1\\
\hline
7 & {\bf 630}&100&33&{\bf 10}&{\bf 3}&1&1&1&1 \\ 
\hline
8 & {\bf 2,520}&430&{\bf 70}&33&{\bf 10}&{\bf 3}&1&1&1 \\
\hline
9 & {\bf 22,680}&1,654&295&103&33&{\bf 10}&{\bf 3}&1&1\\
\hline
10 &{\bf 113,400}&9,033&1,336&247&103&33&{\bf 10}&{\bf 3}&1 \\
\hline
11 & see Thm \ref{Two} &53,549&6,397&998&326&103&33&{\bf 10}&{\bf 3}\\ 
\hline
12& see Thm \ref{Two}&317,728&26,678&4,355&842&330&103&33&{\bf 10}\\
\hline
13& see Thm \ref{Two}&1,642,473&114,720&17,049&3,294&978&330&103&33\\
\hline
14& see Thm \ref{Two}&11,081,916&647,420&81,888&10,709&2,805&1,089&330&103\\
\hline
15& see Thm \ref{Two}&55,409,580&3,887,796&392,033&50,283&8,604&3,144&1,089&330\\
\hline
16& see Thm \ref{Two}&332,457,480&15,551,184&1,898,103&250,867&37,017&9,379&3,399&1,089\\
\hline
17& see Thm \ref{Two}&1,994,744,880&77,755,920&7,592,412&1,261,267&174,655&30,106&10,374&3,399\\
\hline
18& see Thm \ref{Two}&11,968,469,280&514,382,400&33,735,870&3,783,801&862,566&129,756&31,779&10,758\\
[1ex] \hline
\end{tabular}
\label{LowerBounds} 
}
\end{table} 

We have, $P(n,d)=1$, for all $d \ge n$, as a single permutation is a $(n,d)$-PA.
That $P(n,n-1) = 3$, for all $n \ge 3$ was shown in \cite{Klove10}. We show $P(n,n-2) = 10$, for all $n \ge 5$ by Corollary \ref{ten} and the $Clique$ approach ( \cite{Klove10} incorrectly gave $P(n,n-2) \le 9$ ). 
The bound $P(4,2)=6$ was cited in \cite{Klove10}.

We show in Theorem \ref{Two} that $P(n,2) = \frac{n!}{2^{\floor{n/2}}}$. $P(6,3) \ge 20$ was cited in \cite{Klove10}. We computed $P(7,4) \ge 33$ by the $Random/Greedy$ strategy, which improved on the previous lower bound of $28$ \cite{Klove10}. It follows from Theorem \ref{thm1} that $P(n,n-3) \ge 33$, for all $n \ge 7$. 

The bounds $P(7,3) \ge 100$, $P(8,4) \ge 70$, and $P(9,5) \ge 103$ were found by the $Random/Greedy$ strategy, whereas \cite{Klove10} gave lower bounds of $84, 70$ and $95$, respectively.
That $P(n,n-4) \ge 103$, for all $n \ge 9$ follows from Theorem \ref{thm1}.
The bounds $P(8,3) \ge 430, P(9,4) \ge 295, P(10,5) \ge 247, P(11,6) \ge 326$ and $P(12,7) \ge 330$ were all found by the $Random/Greedy$ strategy, whereas \cite{Klove10} gave lower bounds of $401, 283, 236, 236$ and $236$, respectively. That $P(n,n-5) \ge 330$, for all $n \ge 12$, follows from Theorem \ref{thm1}.
The bounds $P(9,3) \ge 1,654, P(10,4) \ge 1,336, P(11,5) \ge 998, P(12,6) \ge 842$ and $P(13,7) \ge 978$ were all found by the $Random/Greedy$ strategy and $P(14,8) \ge 1,089$ was obtained by Theorem \ref{recursion}. That $P(n,n-6) \ge 1,089$, for all $n \ge 14$ follows from Theorem \ref{thm1}. The bounds $P(10,3) \ge 9,033, P(11,4) \ge 6,397, P(12,5) \ge 4,355, P(13,6) \ge 3,294, P(14,7) \ge 2,805, P(15,8) \ge 3,144$ were all found by Theorem~ \ref{prefix}. $P(16,9) \ge 3,399$ was found by Theorem \ref{general}, using $P(9,5)$ and $P(7,4)$.

Theorem \ref{1step} was used to obtain the current lower bound  $P(11,3) \ge 53,549$.
That is, by computation we found $P_3([1..11]-\{3\}) = P_3([1..11]-\{9\} \ge 17,573$ and $P_3([1..11]-\{6\}) \ge 18,403$. 
So, $P(11,3) \ge 2*17,573 + 18,403 = 53,549$. Here is a proof of Theorem \ref{1step}.

\bigskip

{\bf Theorem
\ref{1step}}. Let $A$ be a subset of $[1...(n+1)]$ such that $d(A) \ge d$.
If $n > d \ge 1$, then
$P(n + 1,d) \ge \sum_{i \in A}   P_d([1..(n+1)]-\{i\})$.

\begin{proof}
Let $A=\{a_1,a_2,...,a_k\}$ be a subset of $[1...(n+1)]$ such that $d(A) \ge d$. 
For $a_i \ne a_j$, and permutations $\sigma$ and $\tau$ in $[1..(n+1)]-\{a_i\}$ and $[1..(n+1)]-\{a_j\}$, respectively, $a_i\sigma$ and $a_j\tau$ are permutations on [1..(n+1)] such that $d(a_i\sigma,a_j\tau) \ge d$. 
It follows that $\bigcup_{a_i \in A} a_iB$, with $B$ a set of permutations over $[1..(n+1)]-\{a_i\}$ with Chebyshev distance $\ge d$, is a set of permutations on $[1..(n+1)]$ with Chebyshev distance $\ge d$.
\end{proof}

Here is a proof for Theorem \ref{general}.

{\bf Theorem \ref{general}}.
$P(n,d) \ge \max\{P(n_1,d_1)\cdot  P(n_2,d_2)~|~ d_1+d_2=d$ and $n_1+n_2=n$ and, for some constant $a$, $n_1 = ad_1+r_1$ and $n_2 = ad_2+r_2$, with $0 \le~r_1 \le~d_1$ and
with $0 \le~r_2 \le~d_2\}$, where the maximum is taken over all possible values of ${n_1,n_2,d_1,d_2}$.

\begin{proof} 

Let $n=n_1+n_2$ and $d=d_1+d_2$. 
Let $A$ be a PA on the $n_1$ symbols in $\Sigma_1 = [1 ... n_1]$ with Hamming distance $d_1$ and let $B$ be a PA on the $n_2$ symbols in $\Sigma_2 = [1  ... , n_2]$ with Hamming distance $d_2$. 
Let $\Sigma = [1 ...  n=n_1+n_2 ]$.
Define the function $F_1$ mapping $\Sigma_1$ into $\Sigma$ by: 

$F_1(x) = 
\begin{cases}
x& \mathrm{if~} 1 \le x \le r_1, \\
x+sd_2& \mathrm{if~} (s-1)d_1+r_1+1~ \le x \le~sd_1+r_1, \mathrm{~for~some~} 1 \le s \le a.
\end{cases}$

and define the function $F_2$ mapping $\Sigma_2$ into $\Sigma$ by:

$F_2(x) = 
\begin{cases}
x+(t-1)d_1+r_1& \mathrm{if~} (t-1)d_2 < x \le td_2, \mathrm{~for~some~} 1 \le t \le a,\\
x+n_1, & \mathrm{if~} ad_2~ < x \le ~ad_2+r_2.
\end{cases}$

\bigskip

Construct the PA C = \{~$F_1(\sigma)F_2(\tau)~ |~ \sigma \in A$ and $\tau \in B$~\}. 

C is a set of $|A|\cdot|B|$ permutations on the alphabet $\Sigma$ of $n$ symbols. We show that the Chebyshev distance between permutations in C is at least $d=d_1+d_2$.
Consider two different permutations $\pi_1 = F_1(\sigma_1)F_2(\tau_1)$ and $\pi_2 = F_1(\sigma_2)F_2(\tau_2)$
in C, where $\sigma_1,\sigma_2
\in A$ and $\tau_1,\tau_2 \in B$.  Since $\pi_1 \ne \pi_2$, either $\sigma_1 \ne \sigma_2$ or $\tau_1 \ne \tau_2$.
Due to the similarity of the argument we only explicitly examine the case when $\sigma_1 \ne \sigma_2$. So, the Chebyshev distance between $\sigma_1$ and $\sigma_2$ is at least $d_1$. That is, there is a position $i$ ($1 \le i \le n_1$) such that  $|\sigma_1(i)-\sigma_2(i)| \ge d_1$.
Assume, without loss of generality, that $\sigma_1(i) > \sigma_2(i)$. In other words, $\sigma_1(i)$ and $\sigma_2(i)$ are in different intervals of $d_1$ symbols in $\Sigma_1$, i.e. $\sigma_2(i)$ is in the interval [$(s-1)d_1+r_1,sd_1+r_1$], for some s,
and $\sigma_1(i)$ is in the interval [$(s'-1)d_1+r_1,s'd_1+r_1$], for some $s' > s$. Hence, $F_1$ maps $\sigma_1(i)$ to $\sigma_1(i)+s'd_2$
and maps $\sigma_2(i)$ to 
$\sigma_2(i)+sd_2$. So,
$\vline (\sigma_1(i)+s'd_2)\ - (\sigma_2(i)+sd_2)\vline =  \vline \sigma_1(i) - \sigma_2(i) + s'd_2 - sd_2 \vline = ~\vline \sigma_1(i) - \sigma_2(i)\vline ~+ ~\vline s'd_2 - sd_2 \vline ~\ge~ d_1+d_2$.

\end{proof}

\begin{example}

 For the example $P(16,9) \ge P(9,5)*P(7,4) \ge 3,399$, we see that

$F_1(x) = 
\begin{cases}
x& \mathrm{if~} 1 \le x \le 4, \\
x+4& \mathrm{if~} 5~ \le x \le 9
\end{cases}$

and

$F_2(x) = 
\begin{cases}
x+4& \mathrm{if~} 1 \le x \le 4, \\
x+9& \mathrm{if~} 5~ \le x \le 7
\end{cases}$.

Consider two permutations, say $\rho$ =  1,2,3,4,5,6,7,8,9~and $\sigma$ = 6,1,4,3,2,5,8,9,7, which are at Chebyshev distance 5,
and a permutation, say $\tau$ = 1,2,3,4,5,6,7.
Then, 
$F_1(\rho)$= 1,2,3,4,9,10,11,12,13 and
$F_1(\sigma)$= 
10,1,4,3,2,9,12,13,11. So,\\

$F_1(\rho)F_2(\tau)$ =
1,2,3,4,9,10,11,12,13,5,6,7,8,14,15,16,
and 

$F_1(\sigma)F_2(\tau)$= 
10,1,4,3,2,9,12,13,11,5,6,7,8,14,15,16 \\
are permutations on [1..16] and at Chebyshev distance 9.
\end{example}

Using the construction given in Theorem \ref{general}, we can obtain a PA for $P(3n,3)$ from PAs for $P(2n,2)$ and $P(n,1)$, respectively, which is of size $P(2n,2)*P(n,1)$. As we show in Corollary \ref{lowerbound_2} that $P(2n,2) \ge \frac{(2n)!}{2^{ n}}$ and, clearly, $P(n,1) = n!$, we have, for example, the lower bound $P(3n,3) \ge \frac{(2n)!n!}{2^{n}}$.

Turning now to the specific case of d=2. We first 
prove a recursive lower bound for 
P(n,2).

\begin{theorem}
\label{recursion_2}
For all $n \ge 4$, $P(n,2) \ge P(n-2,2) \binom{n}{2}$. 
\end{theorem}

\begin{proof}

Let $A$ be a PA on the $n-2$ symbols 
$\{1, ... , n-2 \}$ with Chebyshev distance 2. Take new symbols $a=n-1$, $b=n$, and insert them into each permutation of $A$ in each of the possible $\binom{n}{2}$ positions such that $a$ precedes $b$. If in the resulting permutation, the symbols appear in the order $a, n-2, b$, possibly separated by other symbols, then swap the positions of $a$ and $b$. Let the resulting PA be $B$. Clearly, $B$ has $\binom{n}{2}$ times as many permutations as $A$. We show that $B$ has Chebyshev distance 2.

For a proof by contradiction, assume $\sigma , \tau \in B$ have $d(\sigma,\tau) \le 1$. If $\sigma , \tau$ are such that, $\sigma(i),\tau(i) \in \{ a,b \}$ and $\sigma(j),\tau(j) \in \{ a,b \}$, for some $i,j$, then, $d(\sigma,\tau) \ge 2$, because removing symbols a,b gives a permutation in A and all permutations in A have distance at least 2. It follows that two permutations $\sigma , \tau$ have at most one position, say $i$, such that $\sigma (i), \tau (i) \in \{ a,b \}$. If there is no position $i$ such that $\sigma (i), \tau (i) \in \{ a,b \}$, then $d(\sigma,\tau) \ge 2$, as the symbol $b$ is at distance at least 2 with all symbols except $a$ and itself. Similarly, it follows that there cannot be a position $i$ such that $\sigma (i) = \tau (i) = a$ or $\sigma (i) = a$ and $\tau (i) = b$, as this means  $\sigma (j) = b$, for some $j$, and $\tau(j) \notin \{ a ,b \},~ i.e.~ |\sigma(j) - \tau(j) | \ge 2$. 

There is one remaining case, namely, $\sigma (i) = \tau (i) = b$, for some $i$, then, for some $j \ne k$, $\sigma (j) = a$ and $\tau (k) = a$. 
As we are assuming $d(\sigma,\tau) \le 1$, we must have $\tau(j) = n-2$ and $\sigma(k) = n-2$. 
Now consider the order of the positions $i,j$, and $k$. 
If both $j$ and $k$ are less than $i$, say in the order $j < k < i$. 
Then, the permutation $\sigma$ has symbols in the order $a, n-2, b$, which contradicts the requirement that the symbols $a$ and $b$ are swapped. 
If both $j$ and $k$ are greater than $i$, say in the order $i < j < k$, then the permutation $\sigma$ has the symbols in the order $b$, $a$, $n-2$, which contradicts the requirement that the symbols $a$ and $b$ not be swapped. Lastly, if we have the order, say $j < i < k$, then the permutation $\sigma$ has the symbols in the order $n-2$, $b$, $a$, which contradicts the requirement that the symbols $a$ and $b$ not be swapped.
\end{proof}

The following gives a lower bound for $P(n,2)$ which is larger than the bound $P(2a,2) \ge \frac{97}{24}(a!)^2$ in \cite{Klove10} by an exponential factor. It is proven by induction using Theorem \ref{recursion_2}.

\begin{cor}
\label{lowerbound_2}
$P(n,2)\ge \frac{n!}{2^{\floor{n/2}}}$.
\end{cor}

\begin{proof}
This is shown by induction on n. First observe that $P(3,2)=3$ and $P(2,2)=1$. For the inductive step, assume $P(n,2) \ge \frac{n!}{2^{\floor{n/2}}}.$ By Theorem \ref{recursion_2}, $P(n+2,2) \ge P(n,2)*\binom{n+2}{2}$. By the inductive hypothesis, we obtain 
$P(n+2,2) \ge \frac{n!}{2^{\floor{n/2}}}
\frac{(n+2)(n+1)}{2}=\frac{(n+2)!}{2^{\floor{(n+2)/2}}}$
\end{proof}

\bigskip

Here is a proof for Theorem \ref{prefix}.

\textbf{Theorem \ref{prefix}} For any $n\ge d\ge 1$,
$P(n+m,d) \ge \max_{A \in Q((n+m),m,d)}~ \sum_{\sigma \in A} P_d(\sigma^C)$.

\begin{proof}
Let $\sigma_1$ and $\sigma_2$ be  permutations of length $m$ over the alphabet $[1 ... n]$ with Chebyshev distance at least d. We call these  \textit{prefixes}. Let $\tau_1$ and $\tau_2$ be  permutations over $\Sigma^{-\sigma_1}_n$ with Chebyshev distance at least $d$. We call these \textit{suffixes}.  The Chebyshev distance  between $\sigma_1 \tau_1$
and $\sigma_1 \tau_2$ is at least $d$ and the Chebyshev distance between $\sigma_1 \tau$
and $\sigma_2 \tau$ is at least $d$, for any $\tau$.
So, for any set $U \in Q(n,m,d)$, the set $\{ ~\sigma\tau~ | ~ \sigma \in U$ and $\tau \in V$, where $V\in Q_d(\Sigma^{-\sigma}_n) \}$, is a PA on n symbols with pairwise Chebyshev distance at least d and has $\sum_{\sigma \in U} P_d(\Sigma^{-\sigma}_n)$ permutations.
\end{proof}


As an example, we show that $P(12,4) \ge 26,678$. Create a graph, say G, whose nodes are all prefixes of length three and whose edges connect such nodes with Chebyshev distance at least four. Furthermore, a node $\sigma$, a prefix of length three, is given the weight $P_4(\Sigma_{14}^{-\sigma})$.
That is, the weight of a node is the maximum number of suffixes for the given prefix. By Theorem 5, the size of
a maximum weighted clique of G is a lower bound for $P(12,4)$. Using a MaxClique solver \cite{MaxClique} \cite{NetX} we obtaind the lower bound 26,678.

\bigskip

We now give a proof for Theorem \ref{gen}. 

\bigskip

{\bf Theorem
\ref{gen}}. 
For any $d\ge 3$ and $k\ge 1$, 
$$P(dk+d-1,d)\ge \left((k+1)^{d} - {k+d-1\choose d-1}\right) P(dk-1,d).$$

\begin{proof} 
Let $\Phi(a_1,a_2,\dots,a_s)$ denote the alphabet $[1..(dk+d-1)]-\{a_1,a_2,\dots,a_s\}$, for $a_1,a_2,\dots,a_s\in [1..(dk+d-1)]$.
By Theorem \ref{1step}, $P(dk+d-1,d)\ge \sum_{a_1\in A_1} P_d(\Phi(a_1))$, where $A_1=\{d-1,2d-1,\dots,kd+d-1\}$. 
Note that $|\Phi(a_1)|=dk+d-2$.
Similarly, for each $\Phi(a_1)$, by Theorem \ref{1step},   $P_d(\Phi(a_1))\ge \sum_{a_2\in A_2} P_d(\Phi(a_1,a_2))$, where $A_2=\{d-2,2d-2,\dots,kd+d-2\}$. 
Note that $|\Phi(a_1,a_2)|=dk+d-3$.
By applying Theorem \ref{1step} $d-1$ times, $P_d(\Phi(a_1,a_2,\dots,a_{d-2}))\ge \sum_{a_{d-1}\in A_{d-1}} P_d(\Phi(a_1,a_2,\dots,a_{d-1}))$, where $A_{d-1}=\{1,d+1,\dots,kd+1\}$. 
Note that \\$|\Phi(a_1,a_2,\dots,a_{d-1})|=dk$.
\begin{equation}
P(dk+d-1,d)\ge \sum_{a_1\in A_1}
\sum_{a_2\in A_2}
\dots \sum_{a_{d-1}\in A_{d-1}} P_d(\Phi(a_1,a_2,\dots,a_{d-1})).
\end{equation}

Note that there are $k+1$ choices for each of the symbols $a_i$, $1 \le i \le d-1$, with the property that any two choices are at distance at least $d$.
Consider a sequence $\alpha=(a_1,a_2,\dots,a_{d-1})$ with $a_i\in A_i, 1 \le i \le d-1$.
We call such a sequence $a_1,a_2,\dots,a_{d-1}$ {\em monotone} if $a_1>a_2>\dots > a_{d-1}$; 
otherwise, the sequence is {\em mixed}. 

So far, we have sequences, such as $\alpha$, of length $d-1$. We now consider sequences of length $d$ obtained by adding an extra symbol to $\alpha$ (at the end). 
Since $|\Phi(a_1,a_2,\dots,a_{d-1})|=dk$, by Theorem \ref{recursion2} $$P_d(\Phi(a_1,a_2,\dots,a_{d-1}),d)\ge k P(dk-1,d).$$ 
That is, the proof of Theorem \ref{recursion2} shows there are always $k$ symbols one can add to the end of such sequences $\alpha$ and preserve distance d. We show that $P_d(\Phi(a_1,a_2,\dots,a_{d-1}))\ge (k+1) P(dk-1,d)$, if the sequence $a_1,a_2,\dots,a_{d-1}$ is mixed. That is, there are always $k+1$ symbols at pairwise distance d to add to the end of $\alpha$, if $\alpha$ is mixed.
Note that, for symbols $x$ and $y$, such that $d(x,y) \ge d$, $d(\alpha x , \alpha y|) \ge d$.

Assume $a_1,a_2,\dots,a_{d-1}$ is mixed.
We construct a sequence $S = s_1, s_2, \dots s_{k+1}$ of elements in $P_d(\Phi(a_1,a_2,\dots,a_{d-1}))$ with $d(s_i,s_{i+1}) \ge d$, for all i. 
Using $S$ we get $k+1$ sequences, say $\tau_1, \tau_2, \dots , \tau_{k+1}$, where $\tau_i$ consists of 
$a_1,a_2,\dots,a_{d-1}$ followed by $s_i$. It follows that $P_d(\Phi(\tau_i)) \ge (k+1)P(dk-1,d)$.

Consider a table $T$ with $d-1$ columns and $k+1$ rows, where row $i$ of $T$ contains the $i^{th}$ element of $A_j$ and column $j$ of $T$, $1 \le j \le d-1$ contains the elements of $A_{d-j}$ in sorted order. In particular, row $i$ and column $j$ of $T$ contains the element $(i-1)d+j$. See Table \ref{Explain} for an example when $d=6$ and $k=5$.

The desired sequence $S = s_1, s_2, \dots , s_{k+1}$ is obtained from Table \ref{Explain} by choosing one element from each row with the property that the element chosen from row $i+1$ must come from a column whose index is at least as large as the index of the column chosen for row $i$. (This is to ensure distance at least $d$.) Also, an element must be chosen from each row in order to get a sequence of length $k+1$. In addition, one cannot choose any of the elements in the sequence $a_1,a_2,\dots, a_{d-1}$, which are already in $\alpha$, and so are numbers deleted from the alphabet, There is one and only one such symbol  in each column. For example, consider the mixed sequence 17, 22, 15, 8, 1 shown (in bold) in Table \ref{Explain} (represented in the table in right-to-left order).
In this example a desired sequence $S$ can be chosen to be 4, 10, 16, 23, 29, 35.
In the mixed sequence 17, 22, 15, 8, 1 we have $a_1=17 < a_2=22$. 

In every mixed sequence $a_1, a_2, \dots, a_{d-1}$
there must be a $j$ such that $a_j \le a_{j+1}$. The desired sequence $S$ can be chosen by taking elements in order in column $d-j-1$ until (but not including) $a_{j+1}$, say in row $i$), followed by elements in column $d-j$ starting in row $i$ and continuing through all remaining rows. This always works as (1) each column has one and only one deleted element and (2) the condition $a_j \le a_{j+1}$ ensures that the deleted element in column $d-j$ occurs in a row with index smaller than $i$.

Observe that, if $a_1,a_2,\dots,a_{d-1}$ is monotone, there is no $j$ such that $a_j < a_{j+1}$. Consequently, there is no way to construct the desired sequence $S$ by moving to a higher index column when a deleted symbol is encountered. That is, the higher index column always has a different deleted symbol in the given row or a latter row.

\begin{table}[]
    \centering
    \begin{tabular}{|c|c|c|c|c|}
 {\bf 1}&2&3&4&5 \\
7&{\bf 8}&9&10&11\\
13&14&{\bf 15}&16&{\bf 17}\\
19&20&21&{\bf 22}&23\\
25&26&27&28&29\\
31&32&33&34&35\\
    \end{tabular}  \caption{An example of a mixed sequence (in bold), for $d=6$ and $k=5$. The sequence 17,22,15,8,1 is shown right-to-left.}
    \label{Explain}
\end{table}

Let $M$ be the set of all sequences $m_j = a_1,a_2,\dots,  a_{d-1}$ with $a_i \in \{d-i,2d-i,\dots,kd+d-i\}$, for all $i, 1 \le i \le d-1$, with the property that, for $j \ne k$, $d(m_j,m_k) \ge d$. 
Map each sequence $m_i = a_1,a_2,\dots$, $a_{d-1}$ to $x =(x_1,x_2,\dots,x_{d-1}) \in [0..k]^{d-1}$
using $$x=(\floor{a_1/d},\floor{a_2/d},\floor{a_3/d},\dots,\floor{a_{d-1}/d}).$$
A sequence $a_1,a_2,\dots,a_{d-1}$ is monotone if and only if $x_1\ge x_2\ge\dots\ge x_{d-1}$.
The number of such vectors $x$ is ${k+d-1\choose d-1}$. (This is the number of ways of choosing a set of $d-1$ elements from $k+1$ sets of $d-1$ indistinguishable items.)
So, the number of monotone sequences $a_1,a_2,\dots,a_{d-1}$ is $n_{mon}={k+d-1\choose d-1}$.
The number of mixed sequences $a_1,a_2,\dots,a_{d-1}$ is $n_{mix}=(k+1)^{d-1}-{k+d-1\choose d-1}$. 
That is, the number of choices for $a_1 \in A_1, a_2 \in A_2, ... , a_{d-1} \in A_{d-1}$ is $(k+1)^{d-1}$, and
\begin{align}
P(dk+d-1,d)&\ge \sum_{a_1\in A_1}
\sum_{a_2\in A_2}
\dots \sum_{a_{d-1}\in A_{d-1}} P_d(\Phi(a_1,a_2,\dots,a_{d-1}))\\
&\ge (k n_{mon}+(k+1) n_{mix}) P(dk-1,d)\\
&\ge \left((k+1)^{d} - {k+d-1\choose d-1}\right) P(dk-1,d).
\end{align}
The theorem follows.
\end{proof}

Lower bounds for $P(n,d)$
are given in Table \ref{LowerBounds}.
The values in bold are 
exact. 
Precise lower bounds for $P(n,2)$ are given in Theorem \ref{fixed}.
Other lower bounds are from Theorems \ref{1step}, \ref{recursion}, 
\ref{general} and \ref{Two}, and from the Random/Greedy algorithm. We offer some side-by-side comparisons with results from Table II in \cite{Klove10} shown  below in parentheses.

\begin{center}
\begin{tabular}{cc}
$P(5,2) \ge 30$ ~(\it{29})&

$P(7,2) \ge 630$  ~ (\it{582})\\

$P(n,n-2) = 10$, for all $n \ge 5$ ~(\it{9})&

$P(7,3) \ge 100$ ~(\it{84})\\

$P(8,3) \ge 430$ ~(\it{401})&

$P(n,n-3) \ge 33$, for all $n \ge 7$ ~(\it{28})\\

$P(8,4) \ge 70$ ~(\it{68})&

$P(9,4) \ge 295$ ~(\it{283})\\

$P(n,n-4) \ge 103$, for all $n \ge 9$ ~ (\it{95})&

$P(10,5) \ge 247$ ~ (\it{236})\\

$P(11,6) \ge 326$ ~(\it{236})&

$P(n,n-5) \ge 330$, for all $n \ge 12$ ~(\it{236})\\

\end{tabular}
\end{center}


\section{Upper Bounds}

\begin{table}[htb]
\caption{Upper Bounds for $P(n,d)$.} 
\centering
\begin{tabular}{|r |r |r |r |r |r |r |r|r|r|c}
\hline
$n/d$ & 2 & 3 & 4 & 5 & 6 & 7&8&9&10\\ [0.5ex] 
\hline
2 &1&1&1&1&1&1&1&1&1\\ 
\hline
3 &3&1&1&1&1&1&1&1&1\\
\hline
4 & 6&3&1&1&1&1&1&1&1\\
\hline
5 & 30&10&3&1&1&1&1&1&1 \\
\hline
6 & 90&20&10&3&1&1&1&1&1\\
\hline
7 & 630&105&35&10&3&1&1&1&1 \\ \hline
8 & 2,520&560&70&56&10&3&1&1&1 \\ 
\hline
9 & 22,680&1,680&378&126&84&10&3&1&1\\
\hline
10 &113,400&12,600&2,100&256&210&100&10&3&1 \\
\hline
11 &see Thm \ref{Two}&92,400&11,550&1,386&462&330&110&10&3\\ 
\hline
12&see Thm \ref{Two}&369,600&34,650&7,920&924&792&495&120&10\\
\hline
13&see Thm \ref{Two}&3,603,600&270,270&72,072&5,148&1,716&1,287&715&130\\
\hline
14&see Thm \ref{Two}&33,633,600&2,102,100&252,252&30,030&3,432&3,003&2,002&910\\
\hline
15&see Thm \ref{Two}&168,168,000&15,765,750&768,768&420,420&19,305&6,435&5,005&3,003\\
[1ex] \hline
\end{tabular}
\label{UpperBounds} 
\end{table} 

We begin with a proof of Theorem \ref{upperbound}, which is an improvement on  Theorem  \ref{Bound}.
\bigskip
 
{\bf Theorem \ref{upperbound}}.
For $1 \leq k \leq d < n$,
\begin{align*}
P(n,d) \le P(n-k, d) \cdot \binom{n}{k}.    
\end{align*}

\begin{proof}
Consider any PA on $n$ symbols with distance $d$. Partition the PA into subsets determined by the positions of the highest $k$ symbols, $\lbrace n-k+1, n-k+2, \dots, n \rbrace$. Two permutations are in the same subset if their highest $k$ symbols occur in the same subset of $k$ positions, though not necessarily with the same symbol in the same position. For example if $n=5, d=2$, and $k=2$, then the permutations $54321$ and $45132$ would be in the same subset since the symbols $4$ and $5$ both occur in positions $1$ and $2$. Observe that there can be at most $\binom{n}{k}$ subsets since that is the number of ways to choose $k$ positions.

Since any two permutations must have distance at least $d$, and there is no way for any pair of the highest $k \leq d$ symbols to satisfy this distance, within a single subset the Chebyshev distance must be satisfied by the remaining $n-k$ symbols, $\lbrace 1, 2, \dots, n-k \rbrace$. Assume each of the $\binom{n}{k}$ subsets contains $P(n-k, d)$ permutations. If we add one additional permutation to the PA, it will belong to exactly one of these subsets. If we take that subset and delete the highest $k$ symbols from each permutation, we are left with a contracted PA on $n-k$ symbols and distance $d$, however it now contains more than $P(n-k, d)$ permutations, giving us a contradiction. Therefore we can have no more than $P(n-k, d) \cdot \binom{n}{k}$ permutations in the original PA.
\end{proof}

Note that the best results from Theorem \ref{upperbound} typically come from choosing $k=d$.

\begin{example}
By Theorem \ref{upperbound}, $P(11,6) \le P(5,6) \binom{11}{6}$. Since $P(5,6) = 1$, this means $P(11,6) \le \binom{11}{6} = 462$. In \cite{Klove10}, Example 3, they gave  $P(11,6) \le 850$.

\end{example}

Again, we turn to d=2.

\begin{cor}
\label{upperbound_2}
$P(n,2) \le \frac{n!}{2^{\floor{n/2}}}$.
\end{cor}

\begin{proof}
This is shown by induction on $n$. First observe that $P(3,2)=3$ and $P(2,2)=1$. For the inductive step, assume $P(n,2) \le \frac{n!}{2^{\floor{n/2}}}.$ By Theorem \ref{upperbound}, $P(n+2,2) \le P(n,2)*\binom{n+2}{2}$. By the inductive hypothesis, we obtain 
$P(n+2,2) \le \frac{n!}{2^{\floor{n/2}}}
\frac{(n+2)(n+1)}{2}=\frac{(n+2)!}{2^{\floor{(n+2)/2}}}$ 
\end{proof}

{\bf Theorem \ref{Two}}.
$P(n,2) = \frac{n!}{2^{\floor{n/2}}}$.

\bigskip

Theorem \ref{Two})
follows directly from Corollaries \ref{lowerbound_2} and \ref{upperbound_2}.

\bigskip

Upper bounds, for small values of $n$ and $d$, shown in Table \ref{UpperBounds}  were computed by determining the largest clique in a ``distance" graph, $i.e.$ a graph with a node for each permutation and an edge between pairs of nodes at distance at least $d$. 
Others are computed by Theorem \ref{upperbound}.
We offer some side-by-side comparisons with results from Table II in \cite{Klove10} shown in parentheses below.

\bigskip

\begin{center}
\begin{tabular}{cc}

$P(4,2) \le 6$ ~(\it{24})&

$P(5,2) \le 30$ ~(\it{120})\\

$P(6,2) \le 90$ ~(\it{720})&

$P(7,2) \le 630$ ~(\it{5040})\\

\end{tabular}

$P(5,3) \le 10$ 

\end{center}

\bigskip
We give next a proof for Theorem \ref{upperbound2}. The basic idea is that if $P(n_0,n_0-k) \le m$, and m is small enough compared to $n_0$, then one can prove that the diagonal in the lower bound table, such as Table \ref{LowerBounds},
{\it i.e.} $P(n,n-k)$, for all $n \ge n_0$, is also $m$. The argument is a counting argument based on the number of potent symbols and the length of the permutation. 

\bigskip

{\bf Theorem} \ref{upperbound2}.
Suppose that $P(n_0,n_0-k)\le m$ such that 
\begin{equation} \label{condition}
2k(m+1) < (n_0+1)(1+\lfloor n_0/(2k-1)\rfloor).
\end{equation}
Then $P(n,n-k)\le m$, for all $n\ge n_0 \ge 2k$.

\begin{proof}

Suppose to the contrary that $P(n,n-k)\ge m+1$, for some $n>n_0$. 
Let $n$ be the smallest such number. 
Let $A = \{ \pi_1,\pi_2,\dots,\pi_{m+1} \}$ be a PA on $n$ symbols with distance $n-k$. 
Let $k_i$ denote the number of potent symbols in position $i$, taken over all permutations in $A$.
Let $z=1+\lfloor n_0/(2k-1) 
\rfloor $, so $n_0 \ge (z-1)(2k-1)$.
We show that $k_i\ge z$, for all $i$. 
Suppose, by symmetry of argument, that $k_1\le z-1$ and (by rearranging permutation order) only 
$\pi_i,~ 1 \le i \le k_1$, have potent symbols in the first position. 
Observe that each permutation has $2k$ potent symbols, {\it i.e.} the symbols in $[1..k] \cup [n-k+1..n]$,
and that, by our assumption, all of the first $k_1$ permutations, 
and only the first $k_1$ permutations, have a potent symbol in position 1. 
So, if there are $z-1$ permutations, each adding $2k-1$ potent symbols 
to some position $j>1$, the total number of potent symbols (other than the one in position 1) is $(2k-1)(z-1)$.
Since the number of positions, namely, $n>n_0$, is greater than $(2k-1)(z-1)$, by the pigeonhole principle,  there is a position $j>1$ where all $\pi_i,~ 1 \le i\le k_1$, do not have potent symbols.  
Merge columns 1 and $j$ and decrease $n$. That is, do the following:

\begin{itemize}
\item for each permutation $\pi_i$,~$1 \le i \le k_1$, exchange the potent symbol in position 1 with the symbol in position $j$.
\item delete the symbol in position 1 in all permutations (they are no longer potent) and appropriately modify the symbols in each permutation so that they are consecutive integers (deletions may have created gaps).
\end{itemize}

The result is a PA of $m+1$ permutations on $n-1$ symbols with Chebyshev distance $n-k$. This  contradicts our choice of $n$ being smallest.

Note that the total number of potent symbols in the PA $A$ is $2k(m+1)$. Since $k_i \ge z$, for all $1 \le i \le n$,  $2k(m+1) \ge nz \ge (n_0+1)(1+\lfloor n_0/(2k-1)\rfloor)$ which contradicts Inequality \ref{condition}. 
\end{proof}

\begin{table}[htb]
\caption{\label{Pnm2-3}Lower bounds for $P(n,m,2)$ (left) and $P(n,m,3)$ (right).  The tight bounds are in bold.} 
\centering
\begin{tabular}
{cc}

\begin{tabular}{|r |r |r |r |r|}
\hline
$n/m$ & 2 & 3 & 4 & 5 \\ [0.5ex] 
\hline
4 & {\bf 4} & {\bf 6} & {\bf 1} & {\bf 
1}\\ 
\hline
5 &6&15&23&30\\
\hline
6 &  {\bf 9} &24&53&78\\
\hline
7 & 12&42&104&234\\
\hline
8 &  {\bf 16} &59&187&479\\
\hline
9 & 20&88&306&979\\ 
\hline
10 &  {\bf 25} &115&478&1,732\\ 
\hline
11&  30 &158&709&3,002\\
\hline
12 & {\bf 36} &202&1,028&4,805\\
\hline
13 &42 &261&1,430&7,490\\ 
\hline
14& {\bf 49} &322&1,953&11,165\\
\hline
15&56&400&2,600&16,291\\
\hline
\end{tabular}

&

\begin{tabular}{|r |r |r |r |r|}
\hline
$n/m$ & 2 & 3 & 4 & 5 \\ [0.5ex] 
\hline
4 & {\bf 2}& {\bf 3}& {\bf 1}& {\bf 1}\\ 
\hline
5 & {\bf 4} &6&6& {\bf 10}\\
\hline
6 &  {\bf 4} &  {\bf 8} &14&19\\
\hline
7 & 6&15&30&49\\
\hline
8 &  {\bf 9} & 24&49&107\\
\hline
9 &  {\bf 9} &  {\bf 27}&78&181\\ 
\hline
10 & 12&40&118&313\\ 
\hline
11&  {\bf 16} & 59 &177&530\\
\hline
12 & {\bf 16} & {\bf 64} &245&817\\
\hline
13 &20&85&333&1,232\\ 
\hline
14& {\bf 25}&116&466&1,838\\
\hline
15& {\bf 25} & {\bf 125} &601&2,620\\
\hline
\end{tabular}

\end{tabular}
\end{table} 


{\bf Corollary \ref{ten}}.
$P(n,n-2) = 10$, for all $n\ge 5$.

\begin{proof}
$P(n,n-2) = 10$, for all $5 \le n \le 11$, by the clique approach. In Theorem \ref{upperbound2}, 
set $n_0=11, k=2$, and $m=10$. 
Then $z=1+\lfloor n_0/(2k-1)\rfloor=4$ and $2k(m+1)=44 < 48=(n_0+1)z$. So, $P(n,n-2) \le 10$, for all $n \ge 11$, follows by Theorem \ref{upperbound2}.
By Theorem \ref{diag} , $P(n,n-2) \ge 10$, for all $n \ge 5$. 
Therefore $P(n,n-2) = 10$, for all $n \ge 5$. 
\end{proof}

\bigskip
Theorem \ref{fixed} states 
that $P(n,d)$ values along the diagonal $n=d+r$ in Table \ref{LowerBounds} are all equal to $c_r$, if $n\ge d_r$, for some constants $c_r$ and $d_r$. 
Corollary \ref{ten} shows that these constants for $r=2$ are $c_2=10$ and $d_2 = 3$.

\section{Prefixes}
\label{pref}
Computed values for $P(n,m,d)$, for $2 \le d \le 5$, $4 \le n \le 15$, and $2 \le m \le 5$ are given in Tables \ref{Pnm2-3}, \ref{Pnm4-5}. 
For example, $P(9,3,4) \ge 15$, as shown in Table \ref{Pnm4-5}, means there is a set of 15 prefix strings of three symbols over the alphabet $[1..9]$ with pairwise Chebyshev distance 4. For example, 
$\{795,451,125,129,165,169,291,512,516,569,691,851,912,916,956\}$ is such a set. Our computations use a modification of the Random/Greedy algorithm to compute $Q(n,m,d)$. These sets are useful in applications of Theorem \ref{prefix} toward obtaining improved lower bounds. Our computed sets are available on our web site.

\begin{theorem} \label{Pnmd1}
If $d\mid n$ and $d\ge m\ge 2$, then 
$P(n,m,d)=(n/d)^m$.
\end{theorem}

\begin{proof}
Let $k=n/d$. 
First, we show that $P(n,m,d)\le k^m$. 
Let $A$ be an array of size $P(n,m,d)$ in $Q(n,m,d)$. 
Map each permutation $\pi$ in $A$ to $[0..k^m-1]$ using $f(\pi)=\sigma$ where 
$\sigma(i)=j$ if $\pi(i)\in [jd+1\dots jd+d-1]$. 
Since $d(A)=d$, map $f$ is injective.
Therefore $P(n,m,d)\le k^m$.

To show the lower bound $P(n,m,d)\ge k^m$, consider set $A\in Q(n,m,d)$ of permutations $\pi$ such that $\pi(i)\in \{i,i+d,\dots,i+(k-1)d\}$ for all $i\in [1..m]$. 
Then $|A|=k^m$ and $d(A)=d$.
The theorem follows.
\end{proof}

\begin{theorem} \label{Pnmd2}
If $d\mid n$ and $d\ge m\ge 2$, then 
$P(n-i,m,d)=(n/d)^m$ for any $i\in [0..d-m]$.
\end{theorem}

\begin{proof}
Let $k=n/d$. 
By Theorem \ref{Pnmd1}, the theorem follows for $i=0$.
Then $P(n-i,m,d)\le k^m$ for $i\ge 1$.

To show lower bound $P(n-i,m,d)\ge k^m$, consider set $A\in Q(n,m,d)$ of all permutations $\pi$ such that $\pi(j)\in \{j,j+d,\dots,j+(k-1)d\}$, for all $j\in [1..m]$. 
All numbers in $\pi$ are $\le m+(k-1)d = kd+m-d = n+m-d\le n-i$.
The theorem follows.
\end{proof}

\begin{table}[htb]
\caption{\label{Pnm4-5} Lower bounds for $P(n,m,4)$ (left) and $P(n,m,5)$ (right). The tight bounds are in bold.} 
\centering
\begin{tabular}
{cc}

\begin{tabular}{|r |r |r |r |r|}
\hline
$n/m$ & 2 & 3 & 4 & 5 \\ [0.5ex] 
\hline
4 & {\bf 1} & {\bf 1} & {\bf 1} & {\bf 1} \\ 
\hline
5 & {\bf 2} & {\bf 3} & {\bf 3} & {\bf 3} \\
\hline
6 & {\bf 4} &6&6&9\\
\hline
7 & {\bf 4} & {\bf 8} &14&18\\
\hline
8 & {\bf 4} & {\bf 8} & {\bf 16} &30\\
\hline
9 & 6&15&28&55\\ 
\hline
10 & {\bf 9} &24&50&97\\ 
\hline
11& {\bf 9} & {\bf 27} &76&174\\
\hline
12 & {\bf 9} & {\bf 27} & {\bf 81} & 234\\
\hline
13&12&41&116&334\\ 
\hline
14& {\bf 16} &58&176&512\\
\hline
15& {\bf 16} & {\bf 64} &243&803\\
\hline
\end{tabular}

&

\begin{tabular}{|r |r |r |r |r|}
\hline
$n/m$ & 2 & 3 & 4 & 5 \\ [0.5ex] 
\hline
4 & {\bf 1} & {\bf 1} & {\bf 1} & {\bf 1} \\ 
\hline
5 & {\bf 1} & {\bf 1} & {\bf 1} & {\bf 1} \\
\hline
6 & {\bf 2} & {\bf 3} & {\bf 3} & {\bf 3} \\
\hline
7 & {\bf 4} &6&6&9\\
\hline
8 & {\bf 4} & {\bf 8} &14&18\\
\hline
9 & {\bf 4} & {\bf 8} & {\bf 16} &30\\ 
\hline
10 & {\bf 4} & {\bf 8} & {\bf 16} & {\bf 32}\\ 
\hline
11& 6&15&28&55\\
\hline
12 & {\bf 9} &24&49&95\\
\hline
13& {\bf 9} & {\bf 27} &77&173\\ 
\hline
14& {\bf 9} & {\bf 27} & {\bf 81} &236\\
\hline
15& {\bf 9} & {\bf 27} & {\bf 81} & {\bf 243}\\
\hline
\end{tabular}

\end{tabular}
\end{table}

\section{Conclusion and Open Problems}

 We have given several new lower and upper bounds (See Tables \ref{LowerBounds} and \ref{UpperBounds}) for $P(n,d)$ as well as several new techniques for their computation. 
 We conjecture that the bounds for $c_r$ and $d_r$ in Theorem \ref{fixed} can be improved. 
For example, from Table {\ref{LowerBounds}}  it appears that $c_3 \ge  33$ and $c_4 \ge 103$. Is it true that $c_3 = 33, d_3=4$, and $c_4 = 103, d_4=5$?

We computed lower bounds for $P(n,m,d)$ for $n\le 15$ and $m\le 5$ (see Tables \ref{Pnm2-3} and \ref{Pnm4-5}).
The computation of bounds for $P(n,m,d)$ is significantly faster than the computation of bounds for $P(n,d)$ if $m$ is small.
Is there a polynomial time algorithm for computing $P(n,m,d)$, for $m=O(1)$?


\end{document}